\documentclass[11pt,reqno]{amsart}

\usepackage{a4wide}
\usepackage{amsmath,amssymb}
\usepackage{bbm}
\usepackage[dvipsnames]{xcolor}
\usepackage{latexsym}
\usepackage{tikz-cd}
\usepackage{tikz}
\usetikzlibrary{matrix}
\allowdisplaybreaks

\newcommand{\fWAP}{\mathsf{WAP}}
\newcommand{\fSAP}{\mathsf{SAP}}
\newcommand{\cMS}{\mathcal{S}_{\widehat{\mathcal{M}}}}

\usepackage{scalerel,stackengine}
\stackMath

\newcommand\reallywidehat[1]{%
\savestack{\tmpbox}{\stretchto{%
  \scaleto{%
    \scalerel*[\widthof{\ensuremath{#1}}]{\kern-.6pt\bigwedge\kern-.6pt}%
    {\rule[-\textheight/2]{1ex}{\textheight}}
  }{\textheight}%
}{0.5ex}}%
\stackon[1pt]{#1}{\tmpbox}%
}

\newcommand\reallywidecheck[1]{%
\savestack{\tmpbox}{\stretchto{%
  \scaleto{%
    \scalerel*[\widthof{\ensuremath{#1}}]{\kern-.6pt\bigwedge\kern-.6pt}%
    {\rule[-\textheight/2]{1ex}{\textheight}}
  }{\textheight}%
}{0.5ex}}%
\stackon[1pt]{#1}{\scalebox{-1}{\tmpbox}}%
}

\usepackage{hyperref}

\numberwithin{equation}{section}

\newcommand{\Z}{{\mathbb Z}}

\newcommand{\R}{{\mathbb R}}
\newcommand{\N}{{\mathbb N}}
\newcommand{\C}{{\mathbb C}}

\newcommand{\im}{{\mathrm{i}}}
\newcommand{\supp}{{\operatorname{supp}}}

\newcommand{\e}{\operatorname{e}}

\newcommand{\dd}{\mbox{d}}

\newcommand{\cM}{{\mathcal M}}

\newcommand{\cF}{{\mathcal F}}

\newcommand{\cS}{{\textsf S}}

\newcommand{\WAP}{\mathcal{W}\hspace*{-1pt}\mathcal{AP}}
\newcommand{\SAP}{\mathcal{S}\hspace*{-2pt}\mathcal{AP}}

\newcommand{\lm}{\ensuremath{\lambda\!\!\!\lambda}}

\newcommand{\Cu}{C_{\mathsf{u}}}
\newcommand{\Cc}{C_{\mathsf{c}}}
\newcommand{\Cz}{C^{}_{0}}

\theoremstyle{plain}
\newtheorem{theorem}{Theorem}[section]
\newtheorem{proposition}[theorem]{Proposition}
\newtheorem{lemma}[theorem]{Lemma}
\newtheorem{corollary}[theorem]{Corollary}
\newtheorem{fact}[theorem]{Fact}

\theoremstyle{definition}
\newtheorem{definition}[theorem]{Definition}
\newtheorem{remark}[theorem]{Remark}
\newtheorem{example}[theorem]{Example}
\newtheorem{question}[theorem]{Question}

\begin{document}
\title[Translation bounded Fourier transform]{Tempered distributions with translation bounded measure as Fourier transform and the generalized Eberlein decomposition}

\author{Timo Spindeler}
\address{Fakult\"at f\"ur Mathematik, Universit\"at Bielefeld, \newline
\hspace*{\parindent} Postfach 100131, 33501 Bielefeld, Germany}
\email{tspindel@math.uni-bielefeld.de}

\author{Nicolae Strungaru}
\address{Department of Mathematical Sciences, MacEwan University \newline
\hspace*{\parindent} 10700 -- 104 Avenue, Edmonton, AB, T5J 4S2, Canada\\
and \\
Institute of Mathematics ``Simon Stoilow''\newline
\hspace*{\parindent}Bucharest, Romania}
\email{strungarun@macewan.ca}
\urladdr{http://academic.macewan.ca/strungarun/}

\begin{abstract}
In this paper, we study the class of tempered distributions whose Fourier transform is a translation bounded measure and show that each such distribution in $\R^d$ has order at most $2d$.
We show the existence of the generalized Eberlein decomposition within this class of distributions, and its compatibility with
all previous Eberlein decompositions. The generalized Eberlein decomposition for Fourier transformable measures and properties of its components are discussed. Lastly, we take a closer look at the absolutely continuous spectrum of measures supported on Meyer sets.
\end{abstract}



\maketitle

\section{Introduction}

Diffraction is one of the main techniques used by physicists to study the atomic structure of materials. It is used in material sciences to determine the arrangements of atoms in
metals and crystals and in biology to determine the structure of organic compounds.

Mathematically, the process of diffraction is modeled as follows: starting with a set $\Lambda \subset \R^d$, representing the position of atoms in the structure, or more generally a (translation bounded) measure $\mu$, one can construct a measure $\gamma$ called the autocorrelation (or 2-point correlation) of $\Lambda$ or $\mu$ (see \cite{TAO} for the general background). The measure $\gamma$ is positive definite and hence Fourier transformable \cite{ARMA1,BF,MoSt}, and its Fourier transform $\widehat{\gamma}$ is a positive measure \cite{ARMA1,BF,MoSt}. It is this measure $\widehat{\gamma}$ which models the outcome of a physical diffraction experiment, and for this reason it is called the diffraction measure of $\Lambda$ (or $\mu$). It has a decomposition
\[
\widehat{\gamma}=\left(\widehat{\gamma}\right)_{\operatorname{pp}}+\left(\widehat{\gamma}\right)_{\operatorname{ac}}+\left(\widehat{\gamma}\right)_{\operatorname{sc}} \,,
\]
into positive pure point, absolutely continuous and singular continuous measures, respectively.

The pure point spectrum $\left(\widehat{\gamma}\right)_{\operatorname{pp}}$ is usually associated to the order present in the structure. Structures with pure point diffraction, meaning $\widehat{\gamma}=\left(\widehat{\gamma}\right)_{\operatorname{pp}}$, have been of special interest to physicists and mathematicians. While for a log time pure point diffraction has been associated with periodicity, the discovery of quasicrystals \cite{She} has led to a paradigm shift and showed that there is much we do not understand about physical diffraction. It is now known that pure point diffraction is equivalent to the mean almost periodicity of the structure \cite{LSS,LSS2}.

Over the years, it has become clear that structures with non-trivial continuous spectrum can also have interesting properties, and hence are of interest. While the case of structures with pure point diffraction is very well understood now \cite{LSS,LSS2}, the study of systems with continuous or mixed spectrum is just in its infancy. The singular continuous diffraction spectrum $\left(\widehat{\gamma}\right)_{\operatorname{sc}}$ in particular is very mysterious and very little is known about it in general.

\smallskip

One of the most useful tools in the study of the components of the diffraction is the Eberlein decomposition (see below for definition and properties). Given a measure $\mu$ with autocorrelation $\gamma$, there exists a unique decomposition (see Theorem~\ref{prop ebe} and Theorem~\ref{t1} below)
\[
\gamma=\gamma_{\operatorname{s}}+\gamma_0 \,,
\]
called the Eberlein decomposition, with the property that both $\gamma_{\operatorname{s}}$ and $\gamma_0$ are positive definite measures and
\[
\widehat{\,\gamma_{\operatorname{s}}\,}=\left(\widehat{\gamma}\right)_{\operatorname{pp}} \qquad \text{ and } \qquad \widehat{\,\gamma_{0}\,}=\left(\widehat{\gamma}\right)_{\operatorname{c}} \,.
\]
This allows us to study the pure point diffraction spectrum $(\widehat{\gamma})_{\operatorname{pp}}$ and continuous diffraction spectrum $(\widehat{\gamma})_{\operatorname{c}}$ of $\mu$, respectively, in the Fourier dual space by studying instead the components $\gamma_{\operatorname{s}}$ and $\gamma_0$, respectively, of the autocorrelation measure $\gamma$ in the real space. This approach has proven effective in the study of diffraction of measures with Meyer set support \cite{NS1,NS5,NS11,NS20a,NS21} and in the study of compatible random substitution \cite{BSS}.
Because of this, one would like to further decompose
\[
\gamma_{0}= \gamma_{0a}+\gamma_{0s} \,,
\]
such that
\[
\reallywidehat{\gamma_{\operatorname{0s}}}= \left( \widehat{\gamma} \right)_{\operatorname{sc}}  \qquad \text{ and } \qquad  \reallywidehat{\gamma_{\operatorname{0a}}}= \left( \widehat{\gamma} \right)_{\operatorname{ac}}  \,.
\]
Whenever this is possible, we will refer to
\begin{equation}\label{EQged}
\gamma =\gamma_{\operatorname{s}}+\gamma_{0a}+\gamma_{0s}
\end{equation}
as the complete (or generalized) Eberlein decomposition of $\gamma$.

\smallskip

The generalized Eberlein decomposition exists for Fourier transformable measures with Meyer set support \cite{NS20a,NS21}. This result has important consequences for one-dimensional Pisot substitutions, see for example \cite{BaGa,BG2}, and we expect that establishing the existence of the generalized Eberlein decomposition may have important consequences in general. This makes it an important open problem for diffraction theory. Let us state this explicitly.

\begin{question}
Given a positive definite (or more generally a Fourier transformable) measure $\gamma$ on $\R^d$, can one find a decomposition as in \eqref{EQged} which is the Fourier dual to the Lebesgue decomposition of $\widehat{\gamma}\, $?
\end{question}

It is the main goal of the this paper to show that each Fourier transformable measure $\gamma$ on $\R^d$ has a generalized Eberlein decomposition \eqref{EQged} with $\gamma_{0a}, \gamma_{0s}$ tempered distributions of order $2d$.

Our approach is as follows. Since for a Fourier transformable measure $\gamma$, the measure $\widehat{\gamma}$ is translation bounded \cite[Thm. 4.9.23]{MoSt}, we restrict our attention to the class $\mathcal{DTBM}(\R^d)$ of tempered distributions  whose Fourier transform is a translation bounded measure. The Fourier transform then becomes a bijection between $\mathcal{DTBM}(\R^d)$ and $\cM^\infty(\R^d)$. So, we will first study the properties of the space $\mathcal{DTBM}(\R^d)$. In the first main result of the paper, Theorem~\ref{Prop 2}, we show that for each $\omega \in \mathcal{DTBM}(\R^d)$ there exists some $f \in \Cu(\R^d)$ such that
\[
\omega=\left((\partial_{x_1})^{2d}+\ldots+ (\partial_{x_d})^{2d}+(-1)^d\right) (f \lm)
\]
in the sense of distributions. In particular, it follows that each element $\omega \in \mathcal{DTBM}(\R^d)$ is a distribution of order (at most) $2d$.

Next, we prove in Proposition~\ref{gen ebe dist} that each distribution $\omega \in \mathcal{DTBM}(\R^d)$ admits a generalized Eberlein decomposition
\[
\omega=\omega_{\operatorname{s}}+\omega_{\operatorname{0a}}+ \omega_{\operatorname{0s}} \,,
\]
with $\omega_{\operatorname{s}},\omega_{\operatorname{0a}}, \omega_{\operatorname{0s}}\in \mathcal{DTBM}(\R^d)$, which is in Fourier duality with the Lebesgue decomposition.
By combining these results, we prove in Theorem~\ref{coro:mainrd} that each Fourier transformable measure $\gamma$ on $\R^d$ admits a generalized Eberlein decomposition
\[
\gamma=\gamma_{\operatorname{s}}+\omega_{\operatorname{0a}}+ \omega_{\operatorname{0s}} \,,
\]
where $\gamma_{\operatorname{s}}$ and $\gamma_{0}:=\omega_{\operatorname{0a}}+ \omega_{\operatorname{0s}}$ are Fourier transformable measures and $\omega_{\operatorname{0a}}$, $\omega_{\operatorname{0s}}$ are tempered distributions of order (at most) 2d.

We complete the paper by taking a closer look in Section~\ref{sect last} at those distributions in $\mathcal{DTBM}(\R^d)$ whose Fourier transform is pure point, continuous or absolutely continuous measure, respectively.

In general, given a positive definite measure $\gamma$, one would hope all three components of the generalized Eberlein decomposition are measures. It is not clear
if this is the case. More importantly, one can define the diffraction of more general processes (see for  example \cite{LM} and \cite[Sect.~1.4 and Thm.~3.28]{LSS}) in which case the autocorrelation is not necessarily a measure but the diffraction is a translation bounded measure. In this case, the autocorrelation can be viewed as an element in $\mathcal{DTBM}(\R^d)$, and hence our approach can be used in the study of such more general objects.

\smallskip

Lastly, let us mention here that some of the results this paper can be extended to the space $\cMS(\R^d)$ of tempered distributions with measure Fourier transform, which was introduced and studied in \cite{ST}. However, in diffraction theory we only deal with measures whose Fourier transform is a translation bounded measure, and as translation boundedness gives more properties and simplifies many proofs (see for example Theorem~\ref{Prop 2}), we will restrict to $\mathcal{DTBM}(\R^d)$ in this article.

\section{Preliminaries}

Throughout this paper, we will work in the $d$-dimensional Euclidean space $\R^d$. The Lebesgue measure will be denoted by $\lm$ or simply $\dd x$. We will denote by $\Cu(\R^d)$, $\Cc(\R^d)$, $\cS(\R^d)$ and $\Cc^{\infty}(\R^d)$, the spaces of uniformly continuous and bounded functions, compactly supported continuous functions, Schwartz functions and  compactly supported functions which are arbitrarily often differentiable, respectively. For any function $g$ on $\R^d$ and $t \in \R^d$, the functions $T_tg, \widetilde{g}$ and $g^{\dagger}$ are defined by
\[
(T_tg)(x):=g(x-t)\,, \qquad \widetilde{g}(x):=\overline{g(-x)} \qquad \text{ and } \qquad g^{\dagger}(x):=g(-x) \,.
\]

A \emph{measure} $\mu$ on $\R^d$ is a linear functional on $C_{\text{c}}(\R^d)$ such that, for every compact subset $K\subset \R^d$, there is a constant $a_K>0$ with
\[
|\mu(\varphi)| \leqslant a_{K}\, \|\varphi\|_{\infty}
\]
for all $\varphi\in C_{\text{c}}(\R^d)$ with $\supp(\varphi) \subseteq K$. Here, $\|\varphi\|_{\infty}$ denotes the supremum norm of $\varphi$. By Riesz' representation theorem, this definition is equivalent to the classical measure theory concept of regular Radon measure.

Similar to functions, for a measure $\mu$ on $\R^d$ and $t \in \R^d$, we define $T_t\mu, \widetilde{\mu}$ and $\mu^{\dagger}$ by
\[
(T_t\mu)(\varphi):= \mu(T_{-t}\varphi)\,, \qquad \widetilde{\mu}(\varphi):=\overline{ \mu (\widetilde{\varphi})} \qquad  \text{ and } \quad
\mu^{\dagger}(\varphi):= \mu(\varphi^{\dagger}).
\]

\smallskip
Given a measure $\mu$, there exists a positive measure $| \mu|$ such that, for all $\psi \in \Cc(\R^d)$ with $\psi \geqslant 0$, we have \cite{Ped} (compare \cite[Appendix]{CRS2})
\[
| \mu| (\psi)= \sup \{ \left| \mu (\varphi) \right| \ :\ \varphi \in \Cc(\R^d),\,  |\varphi| \leqslant \psi \} \,.
\]
The measure $| \mu|$ is called the \emph{total variation of} $\mu$.

The measure $\mu$ is called \emph{translation bounded} if for all compact sets $K \subset \R^d$ we have
\[
\| \mu \|_{K}:=\sup_{t \in \R^d} \left| \mu \right|(t+K) < \infty \,.
\]
We will denote the space of translation bounded measures by $\cM^\infty(\R^d)$.

\begin{remark}
\begin{itemize}
\item[(a)] A measure $\mu$ is translation bounded if and only if for all $\varphi \in \Cc(\R^d)$ we have \cite{ARMA1,MoSt} $\varphi*\mu \in \Cu(\R^d)$.
Here
\[
(\varphi*\mu)(x):= \int_{\R^d} \varphi(x-y)\ \dd \mu(y) \,.
\]
\item[(b)] A measure $\mu$ is translation bounded if and only if $\| \mu \|_{U} < \infty$ for a single pre-compact Borel set with non-empty interior \cite{BM,SS2}. Moreover, in this case, two different pre-compact Borel sets with non-empty interior define equivalent norms.
\end{itemize}
\end{remark}

Since $\Cc^\infty(\R^d) \subseteq \Cc(\R^d)$, each measure is a distribution of order 0. A measure $\mu$ is called a \emph{tempered measure} if $\mu : \Cc^\infty(\R^d) \to \R^d$ is the restriction to $\Cc^\infty(\R^d)$ of a tempered distribution $\omega : \cS(\R^d) \to \C$. It is easy to see that this is equivalent to $\mu : \Cc^\infty(\R^d) \to \C$ being continuous with respect to the Schwartz topology induced by the embedding $\Cc^\infty(\R^d) \hookrightarrow \cS(\R^d)$ (see for example \cite{ARMA1}).

\begin{remark}
A positive measure is tempered if and only if it is slowly increasing \cite[p.~242]{sch} \cite[Thm.~2.1]{Kaba}, meaning that there exists a polynomial $P \in \R[X_1,\ldots,X_d]$ such that
\[
\int_{\R^d} \frac{1}{1+|P(x)|}\ \dd  \mu(x) < \infty \,.
\]
If $\mu$ is a measure such that $|\mu|$ is slowly increasing, then $\mu$ is tempered \cite[p. 47]{ARMA1}. In this case, we say that $\mu$ is tempered in strong sense.

Note that there is an example of a tempered measure $\mu$ such that $|\mu|$ is not tempered \cite{ARMA1}. In particular, by the above, this measure is tempered but not tempered in strong sense.
\end{remark}

Any translation bounded measure is slowly increasing and hence tempered \cite[Thm.~7.1]{ARMA1}.

\smallskip

The basic tool we will use in this project is the following estimate. As this estimate gives an alternate proof that every translation bounded measure is slowly increasing, we include it here.

\begin{lemma} \label{lem:distr_2}
Let $\nu\in\mathcal{M}^{\infty}(\R^d)$. Then, there is a constant $c>0$ such that
\[
\int_{\R^d} \frac{1}{(2\pi x_1)^{2d}+\ldots+(2\pi x_d)^{2d}+1}\, \dd|\nu|(x) \leqslant c\, \|\nu\|_{[-\frac{1}{2},\frac{1}{2}]^d} < \infty \,.
\]
In particular,
\[
\frac{1}{(2\pi\im (\cdot)_1)^{2d}+\ldots+(2\pi\im (\cdot)_d)^{2d}+(-1)^d}\, \nu
\]
is a finite measure on $\R^d$.
\end{lemma}
\begin{proof}
We first note that
\begin{equation} \label{eq:c_n}
\begin{aligned}
{}
    &\int_{\R^d} \frac{1}{(2\pi x_1)^{2d}+\ldots+(2\pi x_d)^{2d}+1}
       \, \dd|\nu|(x)  \\
    &\phantom{XXXXXX}=\sum_{k\in\Z^d} \int_{k+[-\frac{1}{2},\frac{1}{2}]^d}
       \frac{1}{(2\pi x_1)^{2d}+\ldots+(2\pi x_d)^{2d}+1}\, \dd|\nu|(x)   \\
    &\phantom{XXXXXX}\leqslant \|\nu\|_{[-\frac{1}{2},\frac{1}{2}]^d}
      \sum_{k\in\Z^d} \sup_{x\in k+[-\frac{1}{2},\frac{1}{2}]^d}
      \frac{1}{(2\pi x_1)^{2d}+\ldots+(2\pi x_d)^{2d}+1}   \,.
\end{aligned}
\end{equation}
Now, let $M:=\{k\in\Z^d\,:\, k_j\geqslant 2 \text{ for all } 1\leqslant j\leqslant d\}$. On the one hand, since $\Z^d\setminus M$ is finite, one has
\begin{equation} \label{eq:c_1}
c_1:=\sum_{k\in\Z^d\setminus M} \sup_{x\in k+[-\frac{1}{2},\frac{1}{2}]^d}
 \frac{1}{(2\pi x_1)^{2d}+\ldots+(2\pi x_d)^{2d}+1}  <\infty \,.
\end{equation}
Also note that for $d >1$, by H\"older inequality, we have
\[
\sum_{j=1}^d (2 \pi x_j)^2 \cdot 1 \leqslant \left( \sum_{j=1}^d (2 \pi x_j)^{2d} \right)^\frac{1}{d}  \left( \sum_{j=1}^d 1^q \right)^\frac{1}{q}  \,,
\]
where $q=\frac{d}{d-1}$ is the conjugate of $d$.
Therefore, one has
\begin{equation}\label{eq2}
\left( (2 \pi x_1)^2+ (2 \pi x_2)^2+ \ldots + (2 \pi x_d)^2 \right)^d \leqslant \left( (2 \pi x_1)^{2d}+ (2 \pi x_2)^{2d}+ \ldots + (2 \pi x_d)^{2d} \right) d^{d-1} \,.
\end{equation}
Moreover, \eqref{eq2} trivially holds when $d=1$.
Therefore, we obtain
\begin{equation} \label{eq:c_2}
\begin{aligned}
c_2
    &:=\sum_{k\in M} \sup_{x\in k+[-\frac{1}{2},\frac{1}{2}]^d}
        \frac{1}{(2\pi x_1)^{2d}+\ldots+(2\pi x_d)^{2d}+1}  \\
    &\overset{\eqref{eq2}}{\leqslant} \sum_{k\in M}
     \sup_{x\in k+[-\frac{1}{2},\frac{1}{2}]^d}
        \frac{d^{d-1}}{((2\pi x_1)^{2}+\ldots+(2\pi x_d)^{2})^d}  \\
    &\leqslant \sum_{|k_1|\geqslant 2}\cdots \sum_{|k_d|\geqslant 2}
      \left(\frac{d}{(2\pi (|k_1|-1))^{2}+\ldots+(2\pi (|k_d|-1))^{2}}
      \right)^d      \\
    &\leqslant \left(\sum_{|k_1|\geqslant 1} \frac{d}{(2\pi k_1)^2} \right) \cdot
       \ldots \cdot \left(\sum_{|k_d|\geqslant 1} \frac{d}{(2\pi k_d)^2} \right)
       < \infty \,.
\end{aligned}
\end{equation}
Finally, Eqs.~\eqref{eq:c_n}, \eqref{eq:c_1} and \eqref{eq:c_2} give
\[
\int_{\R^d} \frac{1}{(2\pi x_1)^{2d}+\ldots+(2\pi x_d)^{2d}+1}\, \dd|\nu|(x) \leqslant \|\nu\|_{[-\frac{1}{2},\frac{1}{2}]^d}\, (c_1+ c_2) \,.
\]
The claim follows with $c:=c_1+ c_2$.
\end{proof}

As immediate consequences of Lemma~\ref{lem:distr_2}, we obtain the following results.

\begin{corollary}
Let $\nu\in\mathcal{M}^{\infty}(\R^d)$. Then,
\[
\frac{1}{\left( (2\pi\im (\cdot)_1)^{2}+\ldots+(2\pi\im (\cdot)_d)^{2}\right)^d+(-1)^d}\, \nu
\]
is a finite measure on $\R^d$.
\end{corollary}
\begin{proof}
This follows immediately from Lemma~\ref{lem:distr_2} and the inequality
\[
\frac{1}{\left( (2\pi x_1)^{2}+\ldots+(2\pi x_d)^{2}\right)^d+1} \leqslant  \frac{1}{(2\pi x_1)^{2d}+\ldots+(2\pi x_d)^{2d}+1}  \,.
\]
\end{proof}

\begin{corollary}\label{cor tb implies tempered}\cite[Thm.~7.1]{ARMA1}
Let $\mu \in \cM^\infty(\R^d)$. Then, for all $f \in \cS(\R^d)$, we have
\[
\int_{\R^d} |f(s)|\ \dd |\mu|(s)  < \infty \,.
\]
Moreover, the mapping $f \mapsto \int_{\R^d} f(s)\ \dd\mu(s)$ is a tempered distribution.
\end{corollary}
\begin{proof}
By the above with $P(x_1,\ldots,x_d):=(2\pi x_1)^{2d}+\ldots+(2\pi x_d)^{2d}$, we have
\begin{align*}
\int_{\R^d} |f(x)|\ \dd| \mu|(x)
    &\leqslant  \int_{\R^d} |f(x) (1+ P(x))|\ \dd \Big(\frac{1}{1+P} |\mu|\Big)(x)  \\
    &\leqslant C \| (1+P)f \|_\infty
\end{align*}
for all $f \in \cS(\R^d)$, where
\[
C= \int_{\R^d} \frac{1}{(2\pi x_1)^{2d}+\ldots+(2\pi x_d)^{2d}+1}\, \dd|\nu|(x) < \infty  \,.
\]
The claim follows.
\end{proof}

\medskip

Finally, let us introduce some basic concepts about point sets which we will meet in the paper.

\begin{definition}
A point set $\Lambda\subset\R^d$ is called \emph{relatively dense} if there is a compact set $K\subset \R^d$ such that $\Lambda+K=\R^d$.

The set $\Lambda$ is called \emph{uniformly discrete} if there is an open neighborhood $U$ of $0$ such that $(x+U)\cap (y+U)=\varnothing$ for all distinct $x,y\in \Lambda$.

The set $\Lambda$ is called a \emph{Meyer set} if it is relatively dense and $\Lambda-\Lambda$ is uniformly discrete. Here, $\Lambda - \Lambda$ denotes the \emph{Minkowski difference}
\[
\Lambda - \Lambda := \{ x-y\, :\, x,y \in \Lambda \} \,.
\]
\end{definition}

For more properties of Meyer sets and their full classification, we refer the reader to \cite{LAG,Meyer,MOO}.

\subsection{Fourier transformability}

Next, we briefly review the notion of Fourier transformability for measures and the connection to the Fourier transform of tempered distributions.
Let us recall that the Fourier transform and inverse Fourier transform, respectively, of a function $f \in L^1(\R^d)$ are denoted by $\widehat{f}$ and $\reallywidecheck{f}$. They are defined via
\[
\widehat{f}(x)= \int_{\R^d} e^{-2 \pi i x \cdot y} f(y)\ \dd y  \qquad \text{ and } \qquad
\reallywidecheck{f}(x) = \widehat{f}(-x) \,.
\]

Next, let us review the definition of Fourier transformability for measures.

\begin{definition}
A measure $\mu$ on $\R^d$ is called \emph{Fourier transformable as measure} if there exists a measure $\widehat{\mu}$ on $\R^d$ such that
\[
\reallywidecheck{\varphi}\in L^2(|\widehat{\mu}|) \qquad \text{ and } \qquad
\left\langle \mu\, , \, \varphi*\widetilde{\varphi} \right\rangle =
\left\langle \widehat{\mu}\, , \, |\reallywidecheck{\varphi}|^2 \right\rangle
\]
for all $\varphi \in\Cc(\R^d)$.
In this case, $\widehat{\mu}$ is called the \emph{measure Fourier transform} of $\mu$.
\end{definition}

We will often use the following result.

\begin{theorem}\label{FT measure dist}\cite[Thm.~5.2]{NS20b}
Let $\mu$ be a measure on $\R^d$. Then, $\mu$ is Fourier transformable as a measure if and only if $\mu$ is tempered as a distribution and its Fourier transform as a tempered distribution is a translation bounded measure. Moreover, in this case, the Fourier transform of $\mu$ in the measure and distribution sense coincide.  \qed
\end{theorem}

The following result is easy to prove and will be useful later in the paper. Note that the definition of the Fourier transformability of
measures guarantees that \eqref{eq1} below holds for all Fourier transformable measures $\mu$ and all $\varphi \in K_2(\R^d)=\mbox{Span} \{ \psi * \phi : \psi, \phi \in \Cc(\R^d) \}$. Nevertheless, there are many Fourier transformable measures for which \eqref{eq1} fails for functions $\varphi \in \Cc(\R^d) \backslash K_2(\R^d)$. The fact that $\rho$ is a finite measure is crucial in Proposition~\ref{P1}.

\begin{proposition} \label{P1}
Let $\rho$ be a finite measure on $\R^d$, let $h= \reallywidecheck{\rho}$ and let $\mu=h \lm$. Then, for all
$\varphi \in \Cc(\R^d)$, we have
\begin{equation}\label{eq1}
\mu(\varphi)= \rho(\reallywidecheck{\varphi}) \,.
\end{equation}
\end{proposition}
\begin{proof}
Note here that since $\nu$ is finite, all integrals below are finite and hence, by Fubini's theorem,
\begin{align*}
\mu(f)
    &= \int_{\R^d} \varphi(x)\, h(x) \, \dd x
     = \int_{\R^d} \varphi(x) \int_{\R^d}
       \e^{2\pi\im x\cdot y} \, \dd\rho(y)\, \dd x  \\
    &= \int_{\R^d} \int_{\R^d} \varphi(x)\, \e^{2\pi\im x\cdot y}\, \dd x\,
      \, \dd\rho(y)
     = \int_{\R^d} \reallywidecheck{\varphi}(y)\, \dd \rho(y) = \rho(\reallywidecheck{\varphi})  \,,
\end{align*}
which establishes the desired equality.
\end{proof}

\subsection{Almost periodicity}

Here, we briefly review the concepts of almost periodicity for functions, measures and tempered distributions. For more general overviews, we recommend the review \cite{MoSt} as well as \cite{Eb,LS2,LSS,LSS2,ST}, just to name a few.

\begin{definition}
A function $f\in\Cu(\R^d)$ is \emph{strongly (or Bohr) almost periodic}, if the closure of $\{T_tf\, :\, t\in \R^d\}$ in $(\Cu(\R^d), \| \cdot \|_\infty)$ is compact. $f\in\Cu(\R^d)$ is \emph{weakly almost periodic}, if the closure of $\{T_tf\, :\, t\in \R^d\}$ in the weak topology of the Banach space $(\Cu(\R^d), \| \cdot \|_\infty)$ is compact.
The spaces of strongly and weakly almost periodic functions on $\R^d$ are denoted by $S\hspace*{-1pt}AP(\R^d)$ and $W\hspace*{-2pt}AP(\R^d)$, respectively.
\end{definition}

For every $f\in W\hspace*{-2pt}AP(\R^d)$ the \emph{mean}
\[
M(f):=\lim_{n\to\infty} \frac{1}{(2n)^d} \int_{s+[-n,n]^d} f(x)\ \dd x
\]
exists uniformly in $s\in \R^d$, see \cite[Thm.~14.1]{Eb}, \cite[Prop.~4.5.9]{MoSt}. Moreover, we also have $|f| \in W\hspace*{-2pt}AP(\R^d)$ \cite[Thm.~11.2]{Eb}, \cite[Prop.~4.3.11]{MoSt}.

\begin{definition}
A function $f\in W\hspace*{-2pt}AP(\R^d)$ is \emph{null weakly almost periodic} if $M(|f|)=0$. We denote the space of null weakly almost periodic functions by $WAP_0(\R^d)$.
\end{definition}

In the spirit of \cite{ARMA}, these notions carry over to measures and tempered distributions, respectively, via convolutions with functions in $\Cc(\R^d)$ or $\cS(\R^d)$, respectively.

\begin{definition}
A measure $\mu\in\cM^{\infty}(\R^d)$ is \emph{strongly}, \emph{weakly} or \emph{null weakly almost periodic} if $\varphi*\mu$ is a strongly, weakly or null weakly almost periodic function, for all $\varphi \in\Cc(\R^d)$. We will denote by
$\SAP(\R^d)$, $\WAP(\R^d)$, and $\WAP_0(\R^d)$ the spaces of strongly, weakly and null
weakly almost periodic measures.

Analogously, a tempered distribution $\omega\in\cS'(\R^d)$ is called respectively \emph{weakly, strongly} or \emph{null weakly almost periodic} if, for all $f \in \cS(\R^d)$, the function $f*\omega$ is  respectively weakly, strongly or null weakly almost periodic. Here,
\[
(f*\omega)(x)= \omega( f(x- \cdot)) \in C^\infty(\R^d) \,.
\]
We shall denote the corresponding spaces of tempered distributions by $\fWAP(\R^d), \fSAP(\R^d)$ and $\fWAP_0(\R^d)$ respectively.
\end{definition}

At the end of this section, we introduce the \emph{Eberlein decomposition} for measures and distributions, and discus its relevance to diffraction theory. Let us first remind the reader of the following results.

\begin{proposition}\label{prop ebe}\cite[Thm.~4.10.10]{MoSt}, \cite[Thm.~5.2]{ST}
\begin{itemize}
  \item[(a)] $\displaystyle \WAP(\R^d) = \SAP(\R^d)\oplus \WAP_0(\R^d)$.
  \item[(b)] $\displaystyle \fWAP(\R^d) = \fSAP(\R^d)\oplus \fWAP_0(\R^d)$.
\end{itemize}
\end{proposition}

\medskip

Proposition~\ref{prop ebe} says that each $\mu \in \WAP(\R^d)$ has a unique Eberlein decomposition
\begin{equation}\label{EQ1}
\mu =\mu_{\operatorname{s}}+ \mu_0 \,,
\end{equation}
with $\mu_{\operatorname{s}} \in \SAP(\R^d)$ and $\mu_0 \in \WAP_0(\R^d)$. Similarly, each $\omega \in \fWAP(\R^d)$ has a unique Eberlein decomposition
\begin{equation}\label{EQ2}
\omega =\omega_{\operatorname{s}}+ \omega_0 \,,
\end{equation}
with $\omega_{\operatorname{s}} \in \fSAP(\R^d)$ and $\omega_0 \in \fWAP_0(\R^d)$.

Furthermore, by \cite[Thm.~5.3]{ST}, we have $\WAP(\R^d) \subseteq \fWAP(\R^d)$, $\SAP(\R^d) \subseteq \fSAP(\R^d)$ and $\WAP_0(\R^d) \subseteq \fWAP_0(\R^d)$. In particular, for a measure $\mu \in \WAP(\R^d) \subseteq \fWAP(\R^d)$, the decompositions of \eqref{EQ1} and \eqref{EQ2} coincide.

\medskip

The importance of the Eberlein decomposition for diffraction theory is given by the following two results.

\begin{theorem}\label{t1}\cite[Thm.~4.10.4 and Thm.~4.10.12]{MoSt}
Let $\mu \in \cM^\infty(\R^d)$ be Fourier transformable as measure. Then, $\mu \in \WAP(\R^d)$, the measures $\mu_{\operatorname{s}}$, $\mu_0$ are Fourier transformable and
\[
\reallywidehat{\mu_{\operatorname{s}}}= \left( \widehat{\mu} \right)_{\operatorname{pp}} \qquad \text{ and } \qquad  \reallywidehat{\mu_{0}}= \left( \widehat{\mu} \right)_{\operatorname{c}}  \,.
\]
\end{theorem}

\begin{theorem}\label{t2}\cite[Thm.~6.1]{ST}
Let $\omega \in \cS'(\R^d)$. If $\widehat{\omega}$ is a measure then $\omega \in \fWAP(\R^d)$ and
\[
\reallywidehat{\omega_{\operatorname{s}}}= \left( \widehat{\omega} \right)_{\operatorname{pp}} \qquad \text{ and } \qquad \reallywidehat{\omega_{0}}= \left( \widehat{\omega} \right)_{\operatorname{c}}  \,.
\]
\end{theorem}

As we mentioned in the introduction, it is our goal in this paper to show the existence of the generalized Eberlein decomposition, and study some of its properties. To do this, we will work with the larger class $\mathcal{DTBM}(\R^d)$ of tempered distributions whose Fourier transform is a translation bounded measure. This class contains all measures which are Fourier transformable (as measure).

\section{Properties of $\mathcal{DTBM}(\R^d)$}

As emphasized in the previous section, the following space will play the central role in this section
\[
\mathcal{DTBM}(\R^d):= \{ \omega \in \cS'(\R^d)\, :\, \widehat{\omega} \in \cM^\infty(\R^d) \}  \,.
\]
Note that by Theorem~\ref{t2} we have $\mathcal{DTBM}(\R^d) \subseteq \fWAP(\R^d)$. Moreover, by Theorem~\ref{FT measure dist},
if $\mu$ is a measure on $\R^d$ which is Fourier transformable as measure, then
$\mu \in \mathcal{DTBM}(\R^d)$.

\smallskip

Let us note that every measure $\mu \in \cM^\infty(\R^d)$ is a tempered measure by Corollary~\ref{cor tb implies tempered}, and hence the Fourier transform of some $\omega \in \cS'(\R^d)$, which by definition is in $\mathcal{DTBM}(\R^d)$. Therefore, we get the following simple fact.

\begin{fact}\label{Fact 1}
The Fourier transform is a bijection from $\mathcal{DTBM}(\R^d)$ to $\cM^\infty(\R^d)$.
\end{fact}

Since this space will play a fundamental role in remainder of the paper, we will characterize it in the following. Let us start with a simple characterization for the positive definite tempered distributions in this space. First, recall that $\omega \in \cS'(\R^d)$ is called \emph{positive definite} if for all $f \in \cS(\R^d)$ we have
\[
\omega (f*\widetilde{f}) \geqslant 0 \,.
\]
By the Bochner--Schwartz theorem \cite[Thm. IX.10]{ReSi}, a tempered distribution is positive definite if and only if its Fourier transform is a positive tempered measure.

\medskip

Recall that given $\omega \in \cS'(\R^d)$ and $f \in \cS(\R^d)$, the convolution $f*\omega$
is an infinitely many times differentiable function, which is not necessarily bounded. If $f*\omega$ is bounded for all $f \in \cS(\R^d)$, we say that $\omega$ is a \emph{translation bounded} tempered distribution (see \cite{ST} for properties of these tempered distributions). The space of translation bounded tempered distributions is denoted by $\cS'_{\infty}(\R^d)$.

\medskip

We can now prove the following result.

\begin{proposition}
Let $\omega \in \cS'(\R^d)$ be positive definite. Then, the following statements are equivalent.
\begin{itemize}
\item[(i)] $\omega \in \mathcal{DTBM}(\R^d)$.
\item[(ii)] $\widehat{\omega} \in \cS_{\infty}'(\R^d)$.
\item[(iii)] For all $f \in \cS(\R^d)$ there exists some constant $C>0$ such that
\[
\big|\omega(e^{2 \pi \im t \cdot} f)\big| \leqslant C
\]
for all $t \in \R^d$.
\item[(iv)] There exists some $f \in \cS(\R^d)$ and $C>0$ with $f\neq 0$, $\widehat{f} \geqslant 0$ such that
\[
\big|\omega(e^{2 \pi \im t \cdot} f)\big| \leqslant C
\]
for all $t \in \R^d$.
\end{itemize}
\end{proposition}
\begin{proof}
The equivalence (i)$\iff$(ii) follows from \cite[Prop.~2.5]{ST}.

\medskip

\noindent (ii)$\implies$(iii) Let $f \in \cS(\R^d)$, and let $C:= \| \widehat{f}*\widehat{\omega} \|_\infty < \infty$. The number $C$ is finite because $\widehat{\omega} \in \cS_{\infty}'(\R^d)$.  Then, for all $t \in \R^d$ we have
\[
\big|\omega(e^{-2 \pi \im t \cdot} f)\big|
    = \big|\widehat{\omega}(\reallywidecheck{e^{-2 \pi \im t \cdot} f})\big|
    = \big| (\widehat{f} * \widehat{\omega}) (t) \big| \leqslant C \,.
\]

\smallskip

\noindent (iii)$\implies$(iv) This is obvious.

\medskip

\noindent (iv)$\implies$(i) First, $f \neq 0$ implies $\reallywidecheck{f} \neq 0$. Therefore, there exists some $s \in \R^d$ such that $\reallywidecheck{f}(s) \neq 0$. Since $\reallywidecheck{f}(s)=\widehat{f}(-s) \geqslant 0$, we get $\reallywidecheck{f}(s) >0$. Therefore, there exists some $r>0$ and $c>0$ such that
\[
\reallywidecheck{f}(x) \geqslant c \qquad \text{ for all } x \in B_r(s) \,.
\]

Now, let $\mu := \widehat{\omega}$, which is a positive measure since $\omega$ is positive definite. Since $\reallywidecheck{f}$ and $\mu$ are positive, we have
\begin{align*}
c\, \mu(B_r(x))
    &= \int_{\R^d} c\, 1_{B_r(x)}(y)\ \dd\mu(y)
     = \int_{\R^d} c\, 1_{B_r(s)}(y-x+s)\ \dd\mu(y) \\
    &\leqslant \int_{\R^d} \reallywidecheck{f} (y-x+s)\ \dd \mu(y)
     =\int_{\R^d} \reallywidecheck{\e^{2\pi\im(s-x)\cdot}f}\,(y)\ \dd
       \widehat{\omega}(y) \\
    &=\omega(e^{2 \pi \im (s-x) \cdot} f) \leqslant C
\end{align*}
for all $x \in \R^d$.
This gives that
\[
\| \mu \|_{B_{r}(0)} = \sup_{x \in \R^d} \mu(B_r(x)) \leqslant \frac{C}{c} \,,
\]
which proves the claim.
\end{proof}

Next, by repeating the arguments of \cite{SS2}, we can give the following characterization of $\mathcal{DTBM}(\R^d)$.

\begin{proposition}
Let $\omega \in \cS'(\R^d)$.
Let $B:= \{ \varphi \in \Cc^\infty(\R^d)\, :\, \supp(\varphi) \subseteq [-1,1]^d \}$.
Then, the following statements are equivalent.
\begin{itemize}
\item[(i)] $\omega \in \mathcal{DTBM}(\R^d)$.
\item[(ii)] $\widehat{\omega} \in \cS_\infty'(\R^d)$ and the operator
\[
T: B \to \Cu(\R^d) \,, \qquad  \varphi \mapsto \varphi*\widehat{\omega} \,,
\]
is a continuous operator.
\item[(iii)] There exists a constant $C>0$ such that, for all $t \in \R^d$ and $\varphi \in\Cc^\infty(\R^d)$ with $|\varphi| \leqslant 1_{[-1,1]^d}$, we have
\[
\big|\omega(e^{2 \pi \im t \cdot} \reallywidecheck{\varphi})\big| \leqslant C \,.
\]
\item[(iv)] There exist positive definite $\omega_1, \omega_2, \omega_3, \omega_4 \in \mathcal{DTBM}(\R^d)$ such that
\[
\omega=\omega_1-\omega_2+\im(\omega_3-\omega_4) \,.
\]
\end{itemize}
\end{proposition}
\begin{proof}
(i)$\implies$(iv)  This is immediate. Indeed, if $\mu = \widehat{\omega} \in \cM^\infty(\R^d)$, then there exist positive measures $\mu_1,\mu_2,\mu_3, \mu_4 \in \cM^\infty(\R^d)$ such that
\[
\mu=\mu_1-\mu_2+\im(\mu_3-\mu_4) \,.
\]
As usual, for all $1 \leqslant j \leqslant 4$, there exists some $\omega_j \in \mathcal{DTBM}(\R^d)$ such that $\widehat{\omega_j}= \mu_j$. The claim follows.

\medskip

\noindent (iv)$\implies$(i) This is also immediate, as a (finite) linear combination of translation bounded measures is translation bounded.

\medskip

\noindent (i)$\implies$(ii) Let $\mu =\widehat{\omega} \in \cM^\infty(\R^d)$. Then, $\widehat{\omega} \in \cS_\infty'(\R^d)$ by  \cite[Cor.~2.1]{ST}. Moreover, we have
\[
\|T(\varphi) \|_\infty=\| \varphi*\mu \|_\infty \leqslant \| \mu \|_{ [-1,1]^d } \|\varphi \|_\infty
\]
for all $\varphi \in B$.

\medskip

\noindent (ii)$\implies$(iii) Let $\varphi\in \Cc^\infty(\R^d)$ with $|\varphi| \leqslant 1_{[-1,1]^d}$. Then, $\varphi \in B$ and $\| \varphi \|_\infty \leqslant 1$. Therefore,
\[
\big|\omega(e^{2\pi\im t \cdot} \reallywidecheck{\varphi})\big| =|(\varphi*\widehat{\omega}) (t) | \leqslant \| \varphi* \widehat{\omega} \|_\infty \leqslant  \|T  \|\,  \|\varphi \|_\infty = \|T \| \,.
\]

\smallskip

\noindent (iii)$\implies$(ii) Let $\varphi \in B$. Define
\[
\phi=
\begin{cases}
  \frac{1}{\|\varphi\|_\infty} \varphi & \mbox{ if } \|\varphi \|_\infty \neq 0\,,  \\
  0 & \mbox{ if } \|\varphi \|_\infty = 0 \,.
\end{cases}
\]
Then, $\phi \in B, |\phi| \leqslant 1_{[-1,1]^d}$ and $\varphi= \|\varphi \|_\infty \phi$. Therefore,  we have
\[
| (\varphi*\widehat{\omega})(t)|
    = \|\varphi\|_\infty \left| (\phi*\widehat{\omega})(t) \right|
    =\|\varphi\|_\infty \left|\omega(e^{2 \pi \im t \cdot}\, \reallywidecheck{\phi}
      )\right| \leqslant C \| \varphi \|_\infty
\]
for all $t \in \R^d$.

\medskip

\noindent (ii)$\implies$ (i) Let $N \in \N$ be arbitrary and let $K_N := [-N,N]^d$. Via a standard partition of unity argument, there exist some $\phi_1,\ldots,\phi_k \in \Cc^\infty(\R^d)$ and $t_1,\ldots,t_k \in \R^d$ such that $0 \leqslant \phi_j(x) \leqslant 1$ for all $x \in \R^d$, $\supp(\phi_j)  \subseteq [-1,1]^d$ and
\[
\sum_{j=1}^k T_{t_j}\phi_j (x) = 1
\]
for all $x \in K_N$. Then, for all $\varphi \in \Cc^\infty(\R^d:K_N):=\{ \psi \in \Cc^\infty(\R^d)\, :\, \supp(\psi) \subseteq K_N\}$, we have
\[
\varphi=\sum_{j=1}^k \varphi T_{t_j}\phi_j \,.
\]
Therefore, we obtain
\[
| \widehat{\omega}(\varphi)|
    = | (\widehat{\omega}*\varphi^{\dagger})(0)|
      \leqslant \sum_{j=1}^k | (\widehat{\omega}*(\varphi T_{t_j}\phi_j)^{\dagger})(0)|
      =  \sum_{j=1}^k \left|(\widehat{\omega}*(\phi_jT_{-t_j}\varphi)^{\dagger})
       (t_j) \right|  \,.
\]
Therefore, since $\supp((\phi_jT_{-t_j}\varphi)^{\dagger}) \subseteq [-1,1]^d$, (ii) implies
\begin{align*}
| \widehat{\omega}(\varphi)|
    &\leqslant \sum_{j=1}^k \left|(\widehat{\omega}*(\phi_jT_{-t_j}
       \varphi)^{\dagger}) (t_j) \right|
      \leqslant  \sum_{j=1}^k \|T\|\, \| (\phi_jT_{-t_j}\varphi)^{\dagger}
        \|_\infty  \\
    &\leqslant   \sum_{j=1}^k \|T\|\, \|T_{-t_j}\varphi \|_\infty
      = C_N\, \|\varphi \|_\infty \,,
\end{align*}
where $C_N:= k \|T \|$ depends only on $N$ and the choice of $\phi_1,\ldots,\phi_k$.

Since $\Cc^\infty(\R^d:K_N)$ is dense in $\Cc(\R^d:K_N)=\{ \psi \in \Cc(\R^d) : \supp(\psi) \subseteq K_N \}$, it follows that for all $N$, $\widehat{\omega}$ can be uniquely extended to a continuous functional on $\Cc(\R^d:K_N)$. Therefore, there exists a measure $\mu_N$ supported inside $[-N,N]^d$ such that
\[
\mu_N(\varphi) = \widehat{\omega}(\varphi)
\]
for all $\varphi \in \Cc^\infty(\R^d)$ with $\supp(\varphi) \subseteq [-N,N]^d$. It is easy to see that $\mu_{N}= \mu_{N+1}|_{[-N,N]^d}$. We can then define $\mu: \Cc(\R^d) \to \C$ via
\[
\mu(\psi) = \mu_N(\psi) \qquad \mbox{ with } \supp(\psi) \subseteq [-N,N]^d \,,
\]
and the definition does not depend on the choice of $N$. It is easy to see that $\mu$ is linear, and
\[
| \mu(\psi) | \leqslant C_N \|\psi \|_\infty
\]
for all $\psi \in \Cc(\R^d)$ with $\supp(\psi) \subseteq [-N,N]^d$.
This shows that $\mu$ is a measure and
\[
\mu(\phi)= \widehat{\omega}(\phi) \qquad \text{ for all } \phi \in \Cc^\infty(\R^d) \,.
\]

Finally, setting $\cF:= \{ \psi \in \Cc^\infty(\R^d)\, :\, |\psi| \leqslant 1_{[-1,1]^d} \}$, \cite[Cor.~3.4]{SS2} gives
\[
\| \mu \|_{[-1,1]^d}
    = \sup_{\psi \in \cF} \| \psi*\mu \|_\infty
    =\sup_{\psi \in \cF} \| \psi*\widehat{\omega} \|_\infty
    = \sup_{\psi \in \cF} \| T\psi) \|_\infty \leqslant  \|T \| \,.
\]
This shows that $\mu =\widehat{\omega}$ is a translation bounded measure.
\end{proof}

Next, we give a more explicit description of $\mathcal{DTBM}(\R^d)$ in terms of derivatives of functions in the Fourier--Stieltjes algebra 
\[
B(\R^d):= \{ \widehat{\rho}\, :\, \rho \mbox{ is a finite measure on  } \R^d \} \,. 
\]
In particular, we will show that each tempered distribution $\omega \in \mathcal{DTBM}(\R^d)$ is a distribution of order (at most) $2d$.

\begin{theorem}\label{Prop 2}
Let $\omega \in \mathcal{DTBM}(\R^d)$, and let $\nu=\widehat{\omega}$. Then, there exists some function $h  \in B(\R^d)\subset \Cu(\R^d)$ such that
\begin{itemize}
\item[(a)] $\mu:= h \lm $ is a translation bounded measure and hence a tempered distribution.
\item[(b)] For every $\varphi \in \Cc(\R^d)$, one has
\[
\mu(\varphi) = \int_{\R^d} \frac{\reallywidecheck{\varphi}(x)}{(2\pi\im x_1)^{2d}+\ldots+(2\pi\im x_d)^{2d}+(-1)^d}\ \dd \nu(x) \,.
\]
\item[(c)] Let $\Psi_{d}:=(\partial_{x_1})^{2d}+\ldots+ (\partial_{x_d})^{2d}+(-1)^d$. Then, $\Psi_{d}\mu$ is a tempered distribution and
\[
\Psi_{d}\mu = \omega \,.
\]
\item[(d)] For all $f \in \cS(\R^d)$, we have
\[
\left| \omega (f) \right| \leqslant \|h \|_\infty  \Big( \|f \|_1 + \sum_{j=1}^d \| (\partial_{x_j})^{2d} f \|_1 \Big) \,.
\]
\item[(e)] $\omega$ is a distribution of order $2d$.
\end{itemize}
\end{theorem}
\begin{proof}
 Let $\rho:=\frac{1}{(2\pi\im(\cdot)_1)^{2d}+\ldots+(2\pi\im(\cdot)_d)^{2d}+(-1)^d} \nu$. By Lemma~\ref{lem:distr_2}, $\rho$ is a finite measure, and hence $h:= \reallywidecheck{\rho} \in B(\R^d) \subset \Cu(\R^d)$.

\medskip

\noindent (a) Since $f \in \Cu(\R^d)$, we have $\mu \in \cM^\infty(\R^d)$. In particular, $\mu$ is a tempered distribution.

\medskip

\noindent (b) This follows from Proposition~\ref{P1} with $\rho$ as in the proof of (a).

\medskip

\noindent (c) Let $f \in \cS(\R^d)$. Then,
\begin{align*}
\omega(f)
    &= \widehat{\theta}(\reallywidecheck{f})=\nu(\reallywidecheck{f})
     = \int_{\R^d} \frac{((2\pi\im x_1)^{2d}+\ldots+(2\pi\im x_d)^{2d}+(-1)^d)
         \,\reallywidecheck{f}
       (x)}{(2\pi\im x_1)^{2d}+\ldots+(2\pi\im x_d)^{2d}+(-1)^d}\, \dd\nu(x)  \\
    &= \int_{\R^d} \frac{\reallywidecheck{\Psi_{d}f}\,(x)}{(2\pi\im x_1)^{2d}
        +\ldots+(2\pi\im x_d)^{2d}+(-1)^d} \, \dd\nu(x)
     = \mu(\Psi_{d}f) = (\Psi_{d}\mu)(f) \,.
\end{align*}
Thus, $\omega=\Psi_{d}\mu$ as tempered distributions.

\medskip

\noindent (d) Let $f \in \cS(\R^d)$. Then,
\begin{align*}
\omega(f)
    &= (\Psi_{d}\mu )f= \mu\left(\big((\partial_{x_1})^{2d}+\ldots +
       (\partial_{x_d})^{2d}+(-1)^d\big)f \right) \\
    &=\int_{R^d} h(x) \big(\big((\partial_{x_1})^{2d} + \ldots +
        (\partial_{x_d})^{2d} +(-1)^d\big)f \big)(x)\ \dd \lm (x) \,.
\end{align*}
Therefore,
\begin{align*}
| \omega (f) |
     &\leqslant  \int_{\R^d} \int_{R^d} |h(x)| \big((\partial_{x_1})^{2d}f(x)
       |+\ldots+ |(\partial_{x_d})^{2d}f(x)|+|f(x)| \big) \dd \lm (x) \\
     &\leqslant \|h\|_\infty \Big( \|f \|_1 + \sum_{j=1}^d \|(\partial_{x_j}
        )^{2d} f \|_1 \Big) \,.
\end{align*}

\smallskip

\noindent (e) Let $K$ be a compact subset of $\R^d$, and let $\phi\in\Cc^{\infty}(\R^d)$ with $\text{supp}(\phi)\subseteq K$. Now, (d) and Lemma~\ref{lem:distr_2} imply
\[
\left|\omega(\phi)\right|
    \leqslant \|h\|_\infty \Big( \|\phi \|_1 + \sum_{j=1}^d \|
      (\partial_{x_j})^{2d} \phi \|_1 \Big)
    \leqslant \|h\|_\infty  \,\lm(K) \,\Big(\sum_{j=1}^d\|(\partial_{x_j})^{2d}
      \phi\|_{\infty} + \|\phi\|_{\infty}\Big)  \,.
\]
Thus, $\omega$ is a distribution of order $2d$.
\end{proof}

Let us now look at a very simple example.

\begin{example}
Let $\omega=\delta_{\Z}$. Then, by Poisson's summation formula $\nu:=\widehat{\omega}=\delta_{\Z}$, and hence $\omega \in \mathcal{DTBM}(\R^d)$.
In this case $h(x)=-\frac{1}{2\pi}\sum_{k\in\Z}e^{-\frac{|k-x|}{2\pi}}$.

To see this, we should first remind the reader that
\[
\frac{1}{x^2+1} = \widehat{\phi}(x)  \qquad \text{ with }\ \phi(x) = \frac{1}{2\pi} \e^{-\frac{|x|}{2\pi}} \,.
\]
Now, an application of Poisson's summation formula gives
\begin{align*}
h(x)
    &= -\int_{\R} \e^{-2\pi\im xy} \frac{1}{y^2+1}\, \dd\delta_{\Z}(y)
     = -\sum_{k\in\Z} \frac{1}{k^2+1} \e^{-2\pi\im kx}  \\
    &= -\sum_{k\in\Z} \widehat{\phi}(k) \e^{-2\pi\im kx}
     = -\sum_{k\in\Z} \phi(k-x)
     = -\frac{1}{2\pi}\sum_{k\in\Z}e^{-\frac{|k-x|}{2\pi}} \,.
\end{align*}
\end{example}

\medskip

Let us note in passing that the proofs in this section can be used to prove the following results about tempered distributions with (not necessarily translation bounded) measure Fourier transform.

\begin{lemma}
\begin{itemize}
\item[(a)] Let $h\in B(\R^d)$ and let
\[
\omega=\Psi_{d}(h \lm) \,.
\]
Then, $\omega$ is a tempered distribution and $\nu=\widehat{\omega}$ is a measure satisfying
\[
\int_{\R^d} \frac{1}{(2\pi x_1)^{2d}+\ldots+(2\pi x_d)^{2d}+1}
       \, \dd|\nu|(x) < \infty \,.
\]
\item[(b)] Let $\nu$ be a measure satisfying
\[
\int_{\R^d} \frac{1}{(2\pi x_1)^{2d}+\ldots+(2\pi x_d)^{2d}+1}
       \, \dd|\nu|(x) < \infty \,.
\]
Then, there exists some $h\in B(\R^d)$ such that $\omega=\Psi_{d}(h \lm) $ is a tempered distribution and $\nu=\widehat{\omega}$.
\end{itemize}
\end{lemma}
\begin{proof}
(a) Since $h \in B(\R^d)$, there exists a finite measure $\mu$ on $\R^d$ such that $h=\reallywidecheck{\mu}$.
Let
\[
\nu=\left((2\pi\im(\cdot)_1)^{2d}+\ldots+(2\pi\im(\cdot)_d)^{2d}
       +(-1)^d \right)\mu \,.
\]
Then, $\nu$ is a measure.

Now, since $h=\reallywidecheck{\mu}$ and $\mu$ is a finite measure, we have
\begin{equation}\label{eq1121}
\widehat{h \lm} = \mu
\end{equation}
as measures \cite[Lemma~4.9.15]{MoSt}. In particular, \eqref{eq1121} holds as tempered distributions by Thm.~\ref{FT measure dist}.

Finally, we obtain
\begin{align*}
\widehat{\omega}
    &=\reallywidehat{\Psi_{d}(h \lm)}
     =\left( (2\pi\im(\cdot)_1)^{2d}+\ldots+(2\pi\im(\cdot)_d)^{2d} +(-1)^d
       \right) \widehat{h \lm}  \\
    &= \left((2\pi\im(\cdot)_1)^{2d}+\ldots+(2\pi\im(\cdot)_d)^{2d}
       +(-1)^d \right) \mu = \nu \,.
\end{align*}
Moreover,
\[
\int_{\R^d} \frac{1}{(2\pi x_1)^{2d}+\ldots+(2\pi x_d)^{2d}+1}
       \, \dd|\nu|(x)= \left| \mu \right|(\R^d) < \infty  \,,
\]
since $\mu$ is a finite measure.

\medskip

\noindent (b) This is similar to the proof of Proposition~\ref{Prop 2}.
\end{proof}

\medskip

Exactly the same way as in Theorem~\ref{Prop 2}, we can also prove the following result. Since the proof is identical to the one of Theorem~\ref{Prop 2}, we skip it.

\begin{proposition}
Let $\omega \in \mathcal{DTBM}(\R^d)$, and let $\nu=\widehat{\omega}$. Then, there exists some function $h\in B(\R^d)\subset \Cu(\R^d)$ such that

\begin{itemize}
\item[(a)] $\mu:= h \lm $ is a translation bounded measure and hence a tempered distribution.
\item[(b)] For every $\varphi\in \Cc(\R^d)$, one has
\[
\mu(\varphi) = \int_{\R^d} \frac{\reallywidecheck{\varphi}(x)}{((2\pi\im x_1)^{2}+\ldots+(2\pi\im x_d)^{2})^d+(-1)^d} \, \dd\nu(x) \,.
\]
\item[(c)]Let $\Phi_{d}:=\left((\partial_{x_1})^{2}+\ldots+ (\partial_{x_d})^{2}\right)^{d}+(-1)^d$. Then, $\Phi_{d}\mu$ is a tempered distribution and
\[
\Phi_{d}\mu = \omega \,.
\]
\item[(d)] For all $f \in \cS(\R^d)$ we have
\[
\left| \omega (f) \right| \leqslant \|h \|_\infty  \Big( \|f \|_1 +  \Big\| \Big( \sum_{j=1}^d (\partial_{x_j})^{2} \Big)^d f \Big\|_1 \Big) \,.
\]
\item[(e)] $\omega$ is a distribution of order $2d$.  \qed
\end{itemize}
\end{proposition}

\section{The existence of a generalized Eberlein decomposition}

We can now prove the existence of generalized Eberlein decomposition at the level of distributions of order $2d$.

We start with the following simple results.

\begin{proposition}\label{gen ebe dist}
Let $\omega \in \mathcal{DTBM}(\R^d)$. Then, there exist unique distributions $\omega_{\operatorname{s}}, \omega_{\operatorname{0a}}$ and $\omega_{\operatorname{0s}} \in \mathcal{DTBM}(\R^d)$ such that $\omega =\omega_{\operatorname{s}}+\omega_{\operatorname{0a}}+\omega_{\operatorname{0s}}$ as well as
\[
\widehat{\,\omega_{\operatorname{s}}\,}
    = (\widehat{\omega})_{\operatorname{pp}} \,,  \qquad
        \widehat{\omega_{\operatorname{0a}}}
    = (\widehat{\omega})_{\operatorname{ac}} \qquad \text{ and } \qquad
        \widehat{\omega_{\operatorname{0s}}}
    = (\widehat{\omega})_{\operatorname{sc}} \,.
\]

Moreover, there exist functions $h_{\operatorname{s}}$, $h_{\operatorname{0a}}$ and $h_{\operatorname{0s}} \in B(\R^d)$  such that
\[
\omega_{\operatorname{s}}=\Psi_{d}\left(h_{\operatorname{s}} \lm\right) \,,\qquad \omega_{\operatorname{0a}}=\Psi_{d}\left(h_{\operatorname{0a}} \lm\right)  \qquad \operatorname{ and } \qquad \omega_{\operatorname{0s}}=\Psi_{d}\left(h_{\operatorname{0s}} \lm\right)  \,,
\]
Finally, with $\omega_{\operatorname{0}}=\omega_{\operatorname{0a}}+\omega_{\operatorname{0s}}$ the decomposition
$$
\omega=\omega_{\operatorname{s}}+\omega_{0}
$$
is the decomposition of \eqref{EQ2}.
\end{proposition}
\begin{proof}
Since $\omega \in \mathcal{DTBM}(\R^d)$ we have $\nu:= \widehat{\omega} \in \cM^\infty(\R^d)$, and hence $\nu_{\operatorname{pp}}, \nu_{\operatorname{ac}}, \nu_{\operatorname{sc}} \in \cM^\infty(\R^d)$ by \cite[Lem.~3.12]{SS2}. The existence of  $\omega_{\operatorname{s}}, \omega_{\operatorname{0a}}$ and $\omega_{\operatorname{0s}} \in \mathcal{DTBM}(\R^d)$ follows now from Fact~\ref{Fact 1}.
The existence of  $h_{\operatorname{s}}$, $h_{\operatorname{0a}}$ and $h_{\operatorname{0s}}$ follows from Theorem~\ref{Prop 2}.
The last claim follows from Theorem~\ref{t2}.
\end{proof}

\begin{remark}
Let $\omega \in \mathcal{DTBM}(\R^d)$ be positive definite and let $\nu =\widehat{\omega}$. We can ask wether or not the distributions $\omega_{\operatorname{s}}$, $\omega_{\operatorname{0a}}$ and $\omega_{\operatorname{0s}}$ can be chosen as measures.
However, for the distributions to be measures, it is necessary that $\nu_{\operatorname{pp}}$, $\nu_{\operatorname{ac}}$ and $\nu_{\operatorname{sc}}$ are weakly almost periodic \cite[Thm.~4.11.12]{MoSt}. So, in general the answer is `no'.
For example, if $\nu = \lm|_{[0, \infty)}$ and $\omega =\reallywidecheck{\nu}$, it is clear that $\omega=\omega_{\operatorname{0a}}$ is not a measure.
Similarly, with $\omega= \reallywidecheck{\delta_{\N}}$, the distribution $\omega=\omega_{\operatorname{s}}$ is not a measure.
\end{remark}

In the case of Fourier transformable measures, Proposition~\ref{gen ebe dist} yields the following consequence.

\begin{theorem} \label{coro:mainrd}
Let $\gamma$ be measure on $\R^d$, which is Fourier transformable as a measure. Then, there exist unique Fourier transformable measure $\gamma_{\operatorname{s}}$, tempered distributions $\omega_{\operatorname{0a}}, \omega_{\operatorname{0s}} \in \mathcal{DTBM}(\R^d)$ of order $2d$ and functions $h_{\operatorname{s}}$, $h_{\operatorname{0a}}$ and $h_{\operatorname{0s}} \in B(\R^d)$  such that
\begin{itemize}
\item[(a)]$\gamma_0:=\omega_{\operatorname{0a}}+\omega_{\operatorname{0s}}$ is a Fourier transformable measure.
\item[(b)] $\gamma= \gamma_{\operatorname{s}}+\omega_{\operatorname{0a}}+ \omega_{\operatorname{0s}}$.
\item[(c)] $\gamma=\gamma_{\operatorname{s}}+\gamma_0$ is the Eberlein decomposition of Proposition~\ref{prop ebe}.
\item[(d)] $\widehat{\,\gamma_{\operatorname{s}}\,}= (\widehat{\gamma})_{\operatorname{pp}}$, $\widehat{\omega_{\operatorname{0a}}}= (\widehat{\gamma})_{\operatorname{ac}}$  and  $\widehat{\omega_{\operatorname{0s}}}= (\widehat{\gamma})_{\operatorname{sc}}$.
\item[(e)] $\gamma_{\operatorname{s}}=\Psi_{d}\left(h_{\operatorname{s}} \lm\right)$, $\omega_{\operatorname{0a}}=\Psi_{d}\left(h_{\operatorname{0a}} \lm\right)$   and $\omega_{\operatorname{0s}}=\Psi_{d}\left(h_{\operatorname{0s}} \lm\right)$.
\item[(f)] If $\gamma$ is positive definite, then $\gamma_{\operatorname{s}},\gamma_0$ are positive definite measures and $\omega_{\operatorname{0a}}, \omega_{\operatorname{0s}}$ are positive definite tempered distributions.
\end{itemize}
\end{theorem}
\begin{proof}
Since $\gamma$ is a positive definite measure, we have $\gamma \in \mathcal{DTBM}(\R^d)$ by Theorem~\ref{FT measure dist}.

Now, let
\[
\gamma=\gamma_{\operatorname{s}}+ \gamma_0
\]
be the usual Eberlein decomposition of $\gamma$ from Theorem~\ref{t1}, and let
\[
\gamma=\omega_{\operatorname{s}}+\omega_{\operatorname{0a}}+\omega_{\operatorname{0s}}
\]
be the decomposition of Proposition~\ref{gen ebe dist}. Then, we have
\[
\widehat{\,\gamma_{\operatorname{s}}\,}
    = \left( \widehat{\gamma} \right)_{\operatorname{pp}}
     = \widehat{\,\omega_{\operatorname{s}}\,} \qquad \text{ and }
       \qquad  \widehat{\,\gamma_{0}\,}
    = \left( \widehat{\gamma} \right)_{\operatorname{c}}
     =\left( \widehat{\gamma} \right)_{\operatorname{ac}}
       +\left( \widehat{\gamma} \right)_{\operatorname{sc}}
     = \widehat{\omega_{\operatorname{0s}}}+\widehat{\omega_{\operatorname{0a}}}
\]
in the sense of tempered distributions. The injectivity of the Fourier transform then gives
\[
\gamma_{\operatorname{s}} =\omega_{\operatorname{s}}  \qquad \text{ and } \qquad
\gamma_{0} = \omega_{\operatorname{0s}}+\omega_{\operatorname{0a}} \,.
\]

Note here that $\omega_{\operatorname{0a}}, \omega_{\operatorname{0s}} \in \mathcal{DTBM}(\R^d)$ by Proposition~\ref{gen ebe dist}. Also, by Theorem~\ref{FT measure dist}, we have $\gamma_{\operatorname{s}} \in \mathcal{DTBM}(\R^d)$. In particular, Theorem~\ref{Prop 2} gives the existence of $h_{\operatorname{s}}$, $h_{\operatorname{0a}}$ and $h_{\operatorname{0s}} \in B(\R^d)$ such that (e) holds.

The injectivity of the Fourier transform for tempered distributions and uniqueness of the Lebesgue decomposition for $\widehat{\gamma}$ gives the uniqueness of $\gamma_{\operatorname{s}}, \gamma_0,\omega_{\operatorname{0a}}, \omega_{\operatorname{0s}}$.

Finally, if $\gamma$ is a positive definite measure, then $\widehat{\gamma}$ is a positive measure and hence so are $(\widehat{\gamma})_{\operatorname{pp}}, (\widehat{\gamma})_{\operatorname{c}}, (\widehat{\gamma})_{\operatorname{ac}},(\widehat{\gamma})_{\operatorname{sc}}$. This gives (f) and completes the proof.
\end{proof}

\medskip 
Now, Proposition~\ref{gen ebe dist} makes the following definitions natural.

\begin{definition}
\begin{align*}
\mathcal{DTBM}_{\operatorname{s}}(\R^d) &:=\{ \omega \in \mathcal{DTBM}(\R^d)\, :\, \widehat{\omega} \mbox{ is pure point} \} \\
\mathcal{DTBM}_{\operatorname{0s}}(\R^d) &:=\{ \omega \in \mathcal{DTBM}(\R^d)\, : \, \widehat{\omega} \mbox{ is  singular continuous } \} \\
\mathcal{DTBM}_{\operatorname{0a}}(\R^d) &:=\{ \omega \in \mathcal{DTBM}(\R^d)\, : \, \widehat{\omega} \mbox{ is absolutely continuous} \} \\
\mathcal{DTBM}_{0}(\R^d) &:=\{ \omega \in \mathcal{DTBM}(\R^d)\, :\, \widehat{\omega} \mbox{ is  continuous} \}
\end{align*}
\end{definition}

As usual, we denote by $\cM^\infty_{\operatorname{pp}}(\R^d)$, $\cM^\infty_{\operatorname{c}}(\R^d)$, $\cM^\infty_{\operatorname{ac}}(\R^d)$, $\cM^\infty_{\operatorname{sc}}(\R^d)$
the spaces of translation bounded pure point measures, translation bounded continuous measures, translation bounded absolutely continuous measures and translation bounded singular continuous measures, respectively.

\begin{proposition}\label{prop: GED}
These spaces just defined have the following properties.
\begin{itemize}
\item[(a)] $\mathcal{DTBM}(\R^d) \subseteq \fWAP(\R^d)$, $\mathcal{DTBM}_{\operatorname{s}}(\R^d) \subseteq \fSAP(\R^d)$, $\mathcal{DTBM}_{0}(\R^d) \subseteq \fWAP_0(\R^d)$.
\item[(b)] $\mathcal{DTBM}(\R^d) =\mathcal{DTBM}_{\operatorname{s}}(\R^d) \oplus \mathcal{DTBM}_{0}(\R^d)$.
\item[(c)] $\mathcal{DTBM}_{0}(\R^d) =\mathcal{DTBM}_{\operatorname{0s}}(\R^d) \oplus \mathcal{DTBM}_{\operatorname{0a}}(\R^d)$.
\item[(d)] $\mathcal{DTBM}_{\operatorname{s}}(\R^d)= \mathcal{DTBM}(\R^d) \cap \fSAP(\R^d)$.
\item[(e)] $\mathcal{DTBM}_{0}(\R^d)= \mathcal{DTBM}(\R^d) \cap \fWAP_0(\R^d)$.
\item[(f)] The Fourier transform induces the following bijections:

\begin{tikzpicture}
  \matrix (m) [matrix of math nodes,row sep=2em,column sep=.4em,minimum width=2em]
  { \mathcal{DTBM}(\R^d) &= \mathcal{DTBM}_{\operatorname{s}}(\R^d) &\oplus \mathcal{DTBM}_{\operatorname{0s}}(\R^d) & \oplus \mathcal{DTBM}_{\operatorname{0a}}(\R^d) \\
    \cM^\infty(\R^d)&=\cM^\infty_{\operatorname{pp}}(\R^d)&\oplus  \cM^\infty_{\operatorname{sc}}(\R^d)&\oplus \cM^\infty_{\operatorname{ac}}(\R^d) \\};
  \path[<->]
    (m-1-1) edge node [left] {$\mathcal{F}$} (m-2-1)
    (m-1-2) edge node [left] {$\mathcal{F}$} (m-2-2)
    (m-1-3) edge node [left] {$\mathcal{F}$} (m-2-3)
    (m-1-4) edge node [left] {$\mathcal{F}$} (m-2-4);
\end{tikzpicture}

\end{itemize}
\end{proposition}
\begin{proof}
(a) is a consequence of Theorem~\ref{t2}, \cite[Prop.~6.1]{ST} and \cite[Prop.~6.2]{ST}. (b) and (c) are an immediate consequence of Proposition~\ref{gen ebe dist}, Fact~\ref{Fact 1} and \cite[Lem.~3.12]{SS2}.
(d) and (e) follow from \cite[Prop.~6.1]{ST} and \cite[Prop.~6.2]{ST}, respectively. (f) is a consequence of Fact~\ref{Fact 1}.
\end{proof}

\section{On the components of the generalized Eberlein decomposition}\label{sect last}

Let $\gamma$ be the autocorrelation of $\omega\in\cM^{\infty}(\R^d)$. Then,  $\gamma \in \cM^\infty(\R^d) \cap \mathcal{DTBM}(\R^d)$.
In order to better understand the pure point, continuous, absolute continuous and singular continuous spectrum, respectively, we need a better understanding of the components
of the generalized Eberlein decomposition of $\gamma$. In particular, we need to understand the spaces $\mathcal{DTBM}_{\operatorname{s}}(\R^d)$, $\mathcal{DTBM}_{0}(\R^d)$, $\mathcal{DTBM}_{\operatorname{0s}}(\R^d)$ and $\mathcal{DTBM}_{\operatorname{0a}}(\R^d)$.

\smallskip

The subspaces $\mathcal{DTBM}_{\operatorname{s}}(\R^d)$, $\mathcal{DTBM}_{0}(\R^d)$ are characterized by Proposition~\ref{prop: GED}. Indeed, for a distribution $\omega \in \mathcal{DTBM}(\R^d)$, one has
\begin{itemize}
\item {} $\omega \in \mathcal{DTBM}_{\operatorname{s}}(\R^d)$ if and only if $f*\omega \in SAP(\R^d)$ for all $f \in \cS(\R^d)$.
\item {} $\omega \in \mathcal{DTBM}_{0}(\R^d)$ if and only if $f*\omega \in WAP_0(\R^d)$ for all $f \in \cS(\R^d)$.
\end{itemize}

While we can't constructively obtain the (restricted) Eberlein decomposition this way, it can be used to check if a given decomposition is the Eberlein decomposition.
Indeed, we have the folowing result, compare \cite[Prop.~5.8.7.]{NS11}.

\begin{lemma}
Let $\omega, \omega_1, \omega_2 \in \mathcal{DTBM}(\R^d)$ be such that $\omega=\omega_1+\omega_2$. Then, $\omega_1=\omega_{\operatorname{s}}$ and $\omega_2=\omega_{0}$ if and only if, for all $f \in \cS(\R^d)$, the following two conditions hold:
\begin{itemize}
\item{} $f*\omega_1 \in SAP(\R^d)$.
\item{} $\displaystyle \lim_{n\to\infty} \frac{1}{(2n)^d} \int_{[-n,n]^d} | (f*\omega_2)(x)|\  \dd x =0$.
\end{itemize}
\end{lemma}
\begin{proof}
$\Longrightarrow$ For all $f \in \cS(\R^d)$, we have $f*\omega_1=f*\omega_{\operatorname{s}} \in SAP(\R^d)$ and $f*\omega_2=f*\omega_{0} \in WAP_0(\R^d)$. Therefore, we have
\[
\lim_{n\to\infty} \frac{1}{(2n)^d} \int_{[-n,n]^d} |( f*\omega_2)(x) |\  \dd x =0 \,.
\]

\smallskip

\noindent $\Longleftarrow$ Let $f \in \cS(\R^d)$. Then  $f*\omega_1 \in SAP(\R^d)$. Moreover, by Proposition~\ref{gen ebe dist}(a), we have $f*\omega_2 \in WAP(\R^d)$ and hence $f*\omega_2 \in WAP_0(\R^d)$ by assumption.
Thus, $f*\omega=f*\omega_1+f*\omega_2$ is the Eberlein decomposition of $f*\omega \in WAP(\R^d)$. Therefore, by \cite[Thm.~5.2]{ST} we have
\[
f*\omega_1 =(f*\omega)_{\operatorname{s}}=f*(\omega_{\operatorname{s}}) \qquad \text{ and } \qquad
f*\omega_2 =(f*\omega)_0=f*(\omega_{0})\,.
\]
As this holds for all $f \in \cS(\R^d)$, we get $\omega_1=\omega_{\operatorname{s}}, \omega_2=\omega_{0}$.
\end{proof}

Consequently, we obtain the following simpler characterisation of $\omega \in \mathcal{DTBM}_{0}(\R^d)$. 

\begin{corollary} 
Let $\omega \in \mathcal{DTBM}(\R^d)$. Then, $\omega \in \mathcal{DTBM}_0(\R^d)$ if and only if  
\[
 \lim_{n\to\infty} \frac{1}{(2n)^d} \int_{[-n,n]^d} | (f*\omega)(x)|\  \dd x =0 
\]
for all $f \in \cS(\R^d)$. \qed
\end{corollary}

Going further into the generalized Eberlein decomposition, the problem becomes much more complicated. The main issues is that the space $\mathcal{DTBM}_{\operatorname{0s}}(\R^d)$ is very mysterious, and there is not much we can say about it right now.

Instead, we will take a closer look at $\mathcal{DTBM}_{\operatorname{0a}}(\R^d)$. Via an application of the Riemann--Lebesgue lemma and Plancherel's theorem, we find necessary conditions and sufficient conditions for a tempered distribution $\omega\in \mathcal{DTBM}(\R^d)$ to belong in this space. Finding a necessary and sufficient condition seems to be a difficult problem, which is likely related to finding an intrinsic characterisation of the \emph{Fourier algebra}
\[
A(\widehat{\R^d})=\{ \widehat{f}\, :\, f \in L^1(\R^d) \} \,.
\]

\smallskip

Let us start with necessary conditions. Here, we follow the approach of \cite[Thm.~27]{SS3}.

\begin{proposition}\label{RL}
Let $\omega \in \mathcal{DTBM}_{\operatorname{0a}}(\R^d)$. Then, for all $f \in \cS(\R^d)$, we have $f * \omega \in  \Cz(\R^d)$.
\end{proposition}
\begin{proof}
Let $g \in L^1_{\operatorname{loc}}(\R^d)$ be such that $ \widehat{\omega}= g \lm$, and let $f \in \cS(\R^d)$. Since $\nu:= g \lm \in \cM^\infty(\R^d)$, we have
\begin{align*}
\int_{\R^d} \Big| \widehat{f}(x)\, g(x) \Big|\ \dd x
    &=\int_{\R^d} \Big| \widehat{f}(x)\, g(x) \Big|\ \dd x \\
    &\leqslant \Big\| \Big(\sum_{j=1}^d(2\pi (\cdot)_j)^{2d}+1 \Big) \widehat{f}
       \Big\|_\infty \int_{\R^d} \frac{1}{\sum_{j=1}^d(2\pi (x)_j)^{2d}+1}\
       \dd|\nu|(x)  \\
    &< \infty
\end{align*}
by Lemma~\ref{lem:distr_2}.
Therefore, $\widehat{f}\, g \in L^1(\R^d)$. The Riemann--Lebesgue lemma now gives
\[
f * \omega = \reallywidecheck{ (\widehat{f} g)} \in \Cz(\R^d) \,.
\]
\end{proof}

\begin{remark}
\begin{itemize}
\item[(a)] The converse of Proposition~\ref{RL} is not true \cite[Rem.~28]{SS3}.
\item[(b)] Let $\omega \in \cS'(\R^d)$ be such that $\widehat{\omega}$ is absolutely continuous (but not necessarily translation bounded). Then, exactly as in Proposition~\ref{RL}, one can show that $\reallywidecheck{\varphi} * \omega \in \Cz(\R^d)$, for all $\varphi \in \Cc^\infty(\R^d)$.
\end{itemize}
\end{remark}

Next, we give a sufficient condition for the Fourier transform of a tempered distribution to be a continuous measure with $L_{\text{loc}}^2(\R^d)$-density function.

\begin{proposition}\label{propL2}
Let $\omega \in \cS'(\R^d)$. Then, there exists some $f \in L^2_{\operatorname{loc}}(\R^d)$ such that $\widehat{\omega} =f \lm$ if and only if  $\reallywidecheck{\varphi}*\omega \in L^2(\R^d)$ for all $\varphi \in \Cc^\infty(\R^d)$.
\end{proposition}
\begin{proof}
 $\Longrightarrow$ Let $\varphi \in \Cc^\infty(\R^d)$, and let $K \subset \R^d$ be the support of $\varphi$. Let $g:=f \varphi$.
Then,
\[
\int_{\R^d} |g(x)|^2\ \dd x
    = \int_{K} |f(x)|^2\,|\varphi(x)|^2\ \dd x
    \leqslant \| \varphi \|_\infty^2 \int_{K} |f(x)|^2\ \dd x < \infty \,.
\]
Therefore, $g \in L^2(\R^d)$, and hence, by Plancharel's theorem, there exists some $h \in L^2(\R^d)$ such that $\widehat{h}=g$.

Now, by \cite[Thm.~2.2]{ARMA1}, $h \lm$ is Fourier transformable as a measure and $\widehat{h \lm}= g \lm$. In particular, $h \lm $ is a tempered distribution by Theorem~\ref{FT measure dist}, and as tempered distributions we have
\[
\widehat{h \lm}= g\lm= f \varphi \lm =\reallywidehat{ (\reallywidecheck{\varphi} * \omega)\lm}
\,.
\]
Therefore, the functions $h$ and $\omega * \reallywidecheck{\varphi}$ agree as tempered distributions, and hence agree almost everywhere.

Finally, $h \in L^2(\R^d)$ implies $\omega*\reallywidecheck{\varphi} \in L^2(\R^d)$
as claimed.

\medskip

\noindent $\Longleftarrow$ For each $\varphi \in \Cc^\infty(\R^d)$, we have $\omega * \reallywidecheck{\varphi} \in L^2(\R^d)$. Let $f_{\varphi} \in L^2(\R^d)$ be the $L^2$-Fourier transform of $\omega * \reallywidecheck{\varphi}$. Then, exactly as above, as tempered distributions we have
\begin{equation} \label{eq:vwo}
\varphi \widehat{\omega}= \reallywidehat{(\omega*\reallywidecheck{\varphi}) \lm }= f_{\varphi} \lm \,.
\end{equation}

Next, for each $n\in\Z^d$, fix some $\varphi_n\in\Cc(\R^d)$ such that $\varphi_n(x) =1$ for all $x\in n+[0,1)^d$. Now, define
\begin{equation} \label{eq:ffv}
f(x):= f_{\varphi_n}(x) \qquad \text{ if } x\in n+[0,1)^d \,.
\end{equation}
Then, $\varphi_n\in L^2(\R^d)$ implies $f|_{n+[0,1)^d}\in L^2(\R^d)$. As every compact set be covered by finitely many sets of the form $n+[0,1)^d$, we obtain $f\in L_{\text{loc}}^2(\R^d)$.

Finally, if $\varphi\in\Cc(\R^d)$ such that $\supp(\varphi)\subseteq n+[0,1)^d$ for some $n$, we have $\varphi=\varphi\varphi_n$, and hence by \eqref{eq:vwo} and \eqref{eq:ffv}
\[
\widehat{\omega}(\varphi) = (\varphi_n\widehat{\omega})(\varphi) = (f_{\varphi_n}\lm)(\varphi) = \int_{n+[0,1)^d} f_{\varphi_n}(x)\, \varphi(x)\ \dd x = (f\lm)(\varphi) \,.
\]

Now, via a partition of unity argument, we can write each $\phi \in \Cc^\infty(\R^d)$ as a linear combination of $\varphi \in \Cc^\infty(\R^d)$ with $\supp(\varphi) \subseteq n+[0,1)^d$  and hence $\omega = f \lm$ on $\Cc^\infty(\R^d)$. Therefore, by the density of $\Cc^\infty(\R^d)$ in $\cS(\R^d)$, the two tempered distributions agree.
\end{proof}

Cauchy--Schwarz' or H\"older's inequality imply that $L^2_{\text{loc}}(\R^d) \subseteq L^1_{\text{loc}}(\R^d)$, which immediately leads to the following sufficient conditions.

\begin{corollary}
Let $\omega \in \mathcal{DTBM}(\R^d)$. If $\widehat{\varphi}*\omega \in L^2(\R^d)$ for all $\varphi \in \Cc^\infty(\R^d)$, then $\omega \in \mathcal{DTBM}_{\operatorname{0a}}(\R^d)$. \qed
\end{corollary}

\begin{corollary}\label{cor3}
Let $\omega \in \mathcal{DTBM}(\R^d)$. If $g*\omega \in L^2(\R^d)$ for all $g \in \cS(\R^d)$, then $\omega \in \mathcal{DTBM}_{\operatorname{0a}}(\R^d)$. \qed
\end{corollary}

Consider now a Fourier transformable measure $\mu$ with uniformly discrete support. If we could strengthen Corollary~\ref{cor3} by replacing `for all $g \in \cS(\R^d)$' by `for all $g \in \Cc^\infty(\R^d)$', then a simple computation allows us to replace this condition by a much simpler restriction on the coefficients of $\mu$. This will in turn have interesting consequences for Fourier transformable measures with Meyer set support.

In order to do this, let us introduce the following definition.

\begin{definition}
A measurable function $f : \R^d \to \C$ is called \emph{Schwartz to $L^2$ compatible} if $gf \in L^2(\R^d)$ for all $g \in \cS(\R^d)$.
We denote the space of Schwartz to $L^2$ compatible functions by\footnote{We reversed the order of letters here since $SL_2$ is used for the special linear group.} $\mathcal{LS}_2(\R^d)$.

A measurable function $f : \R^d \to \C$ is called \emph{weakly Schwartz to $L^2$ compatible} if $\widehat{\varphi}f \in L^2(\R^d)$ for all $\varphi \in \Cc^\infty(\R^d)$.
We denote the space of weakly Schwartz to $L^2$ compatible functions by $\mathcal{WLS}_2(\R^d)$.
\end{definition}

Let us start with some straight forward consequences of the definition.

\begin{lemma}\label{lemma1}
\begin{itemize}
\item[(a)] $\mathcal{LS}_2(\R^d) \subseteq \mathcal{WLS}_2(\R^d) \subseteq L^2_{\operatorname{loc}}(\R^d) \subseteq L^1_{\operatorname{loc}}(\R^d)$.
\item[(b)] Let $f \in L^1_{\operatorname{loc}}(\R^d)$. If there exists some $k \in \N$ and some $2 \leqslant p \leqslant \infty$ such that $\frac{1}{(1+|\cdot|^2)^k}\, f\in L^p(\R^d)$, then $f \in \mathcal{LS}_2(\R^d)$.
\item[(c)] $L^\infty(\R^d) \subseteq \mathcal{LS}_2(\R^d)$.
\end{itemize}
\end{lemma}
\begin{proof}
(a) $\mathcal{LS}_2(\R^d) \subseteq \mathcal{WLS}_2(\R^d)$ follows from the fact that $\widehat{\varphi} \in \cS(\R^d)$ for all $\varphi \in \Cc^\infty(\R^d)$.

Next, let $f \in \mathcal{WLS}_2(\R^d)$. Let $K \subseteq \R^d$ be any compact set. It is easy to see that there is some $\varphi \in \Cc^\infty(\R^d)$ such that $\widehat{\varphi} \geqslant 1_{K}$. Then,
\[
\int_{K} | f(x)|^2\ \dd x
    \leqslant   \int_{\R^d} | f(x) |^2\, |\widehat{\varphi}(x)|^2\ \dd x
    < \infty \,.
\]
This shows that $ \mathcal{WLS}_2(\R^d) \subseteq L^2_{\operatorname{loc}}(\R^d)$. As mentioned above, $L^2_{\operatorname{loc}}(\R^d) \subseteq L^1_{\operatorname{loc}}(\R^d)$
follows from Cauchy--Schwarz' or H\"older's inequality.

\medskip

\noindent (b) Let $g \in \cS(\R^d)$, and let $2 \leqslant q  \leqslant \infty$ be so that $\frac{2}{p}+\frac{2}{q}=1$. Such a $q$ exists, since $p\geqslant2$. Then, an application of H\"older's inequality gives
\begin{align*}
\| gf \|^2_2
    &= \Big\| \Big| \frac{1}{(1+|\cdot|^2)^k} f \Big|^2\,  \Big| (1+|\cdot|^2)^k
       g \Big|^2 \Big\|_1 \\
    &\leqslant \Big\| \Big| \frac{1}{(1+|\cdot|^2)^k} f \Big|^2 \Big\|_\frac{p}{2}
       \ \Big\| \Big| (1+|\cdot|^2)^k g\Big|^2 \Big\|_{\frac{q}{2}} \\
    &= \Big\|  \frac{1}{(1+|\cdot|^2)^k} f\Big\|_p\ \Big\| (1+|\cdot|^2)^k g
       \Big \|_q < \infty
\end{align*}
because of the assumption on $f$ and $g \in \cS(\R^d)$.

\medskip

\noindent (c) This follows from (b).
\end{proof}

\begin{remark}
\begin{itemize}
\item[(a)] Let
\[
f(x)=
\begin{cases}
\frac{1}{\sqrt{x}} &\mbox{ if } 0< x \leqslant 1 \,, \\
0 & \mbox{ otherwise}\,.
\end{cases}
\]
Then, $f \in L^1(\R) \subseteq L^1_{loc}(\R)$ and $f \in L^p(\R)$ for all $1 \leqslant p <2$.

Next, if $g \in \cS(\R)$ is any function such that $g \geq 1$ on $[0,1]$, then $fg \notin L^2(\R)$. This shows that $f \notin \mathcal{LS}_2(\R^d) $ and $f \notin \mathcal{WLS}_2(\R^d) $.
In particular, the condition $2 \leqslant p \leqslant \infty$ in Lemma~\ref{lemma1}(b) is sharp.

\item[(b)] Let $f \in L^1_{\operatorname{loc}}(\R^d)$ be such that $f \lm$ is a tempered measure in strong sense. Then, there exists some $n\in\N$ such that
\[
\int_{\R^d} \frac{1}{(1+|x|^2)^n}\, |f(x)|\ \dd x < \infty \,,
\]
which is the condition from Lemma~\ref{lemma1}(b) for $p=1$. Unfortunately, as pointed above, this is not enough to conclude that $f \in \mathcal{LS}_2(\R^d)$.
\end{itemize}
\end{remark}

\medskip

The next result is an immediate consequence of Lemma~\ref{lemma1}.

\begin{corollary}\label{Cor212}
We have the following equality of spaces.
\begin{align*}
L^{\infty}(\R^d)
    &=L^1_{\operatorname{loc}}(\R^d)\, \cap\, L^\infty(\R^d)
      =L^2_{\operatorname{loc}}(\R^d) \, \cap\, L^\infty(\R^d) \\
    &= \mathcal{WLS}_2(\R^d)\, \cap\, L^\infty(\R^d)= \mathcal{LS}_2(\R^d)
       \, \cap\, L^\infty(\R^d)   \,.
\end{align*}
\end{corollary}
\begin{proof}
Lemma~\ref{lemma1}(c) implies
\[
L^\infty(\R^d) \subseteq \mathcal{LS}_2(\R^d) \subseteq \mathcal{WLS}_2(\R^d) \subseteq L^2_{\operatorname{loc}}(\R^d) \subseteq L^1_{\operatorname{loc}}(\R^d) \,.
\]
 The claim follows.
\end{proof}

Now, we can prove the following result. Note here that the only difference between this result and Proposition~\ref{propL2}, is that the definition of $\mathcal{LS}_2(\R^d)$ and $\mathcal{WLS}_2(\R^d)$, respectively, allows us replace in the proof of Proposition~\ref{propL2} the function $\reallywidecheck{g}$ for $g \in \Cc^\infty(\R^d)$ by $h \in \cS(\R^d)$ and $\varphi \in \Cc^\infty(\R^d)$, respectively. Repeating the proof with these changes, we get the following result. As the conversion of the proof is straight forward, we skip it.

\begin{proposition}\label{T2}
Let $\omega \in \cS'(\R^d)$.
\begin{itemize}
\item[(a)] There exists some $f \in \mathcal{LS}_2(\R^d)$ such that $\widehat{\omega}=f \lm$ if and only if $h * \omega \in L^2(\R^d)$ for all $h \in \cS(\R^d)$.
\item[(b)] There exists some $f \in \mathcal{WLS}_2(\R^d)$ such that $\widehat{\omega}=f \lm$ if and only if $\varphi * \omega \in L^2(\R^d)$ for all $\varphi \in \Cc^\infty(\R^d)$.
\end{itemize}
\qed
\end{proposition}

\medskip
We can now prove the following result, which has interesting consequences for measures with Meyer sets support.

\begin{theorem}\label{T4} Let $\Lambda \subset \R^d$ be a uniformly discrete set, and let
\[
\mu= \sum_{x \in \Lambda} c_x \delta_x
\]
be a tempered measure. Then, the following statements are equivalent.
\begin{itemize}
\item[(i)] There exists some $f \in \mathcal{WLS}_2(\R^d)$ such that $\widehat{\mu}=f \lm$.
\item[(ii)] $\varphi * \mu \in L^2(\R^d)$ for all $\varphi \in \Cc^\infty(\R^d)$.
\item[(iii)] $\sum_{x \in \Lambda } |c_x|^2 < \infty$.
\end{itemize}
\end{theorem}
\begin{proof}
(i)$\iff$(ii): This is a consequence of Proposition~\ref{T2}.

\medskip

\noindent (ii)$\implies$ (iii): Let $U \subseteq \R^d$ be such that $\Lambda$ is $U$-uniformly discrete. Pick $\varphi \in \Cc^\infty(\R^d)$ such that $\supp(\varphi) \subseteq U$ and $\phi \geqslant 1$ on some compact set $K \subseteq U$ of positive Lebesgue measure. Then, by \cite[Lem. 5.8.3]{NS11}, we have
\[
| \mu* \varphi |= | \mu | * |\varphi| \,.
\]
Moreover, by the $U$-uniform discreteness of $\Lambda$ (compare the proof of \cite[Lem. 5.8.3]{NS11}), for each $y\in \R^d$, there exists at most one $x \in\Lambda$ such that
$|\varphi(y-x)| \neq 0$. This immediately gives
\[
| (\mu* \varphi) (y)|^2 =\sum_{x \in \Lambda} |c_x|^2\, |\varphi(y-x)|^2 \geqslant \sum_{x \in \Lambda} |c_x|^2\, |1_{K}(y-x)|  \qquad \text{ for all } y \in \Lambda
\,.
\]
Note here that if $\Lambda$ is finite then,
\[
\lambda(K) \sum_{x \in \Lambda} |c_x|^2 = \sum_{x \in \Lambda} |c_x|^2 \int_{\R^d} |1_{K}(y-x)|\ \dd y =\int_{\R^d} \sum_{x \in \Lambda} |c_x|^2 \,  |1_{K}(y-x)|\ \dd y  \,.
\]
On another hand, if $\Lambda$ is infinite, it is countable by uniform discreteness. Let $\{ x_n \}$ be an enumeration of $\Lambda$. Then, the sequence of functions
$f_n:=\sum_{k=1}^n |c_{x_k}|^2 |1_{K}(\cdot-x_k)| \in L^1(\R^d)$ is increasing, and hence
\[
\lambda(K) \sum_{x \in \Lambda} |c_x|^2
    = \lim_{n\to\infty} \int_{\R^d} f_n(y)\ \dd y
    = \int_{\R^d} \lim_{n\to\infty} f_n(y)\ \dd y
    = \int_{\R^d} \left( \sum_{x \in \Lambda} |c_x|^2\, |1_{K}(y-x)| \right) \dd y \,,
\]
where we applied the monotone convergence theorem.
Therefore, in both cases, we obtain
\[
\lambda(K) \sum_{x \in \Lambda} |c_x|^2
    = \int_{\R^d} \left( \sum_{x \in \Lambda} |c_x|^2  |1_{K}(y-x)| \right) \dd y
    \leqslant  \int_{\R^d} | (\mu* \varphi)(y)|^2\  \dd y < \infty  \,.
\]

\smallskip

\noindent (iii)$\implies$(ii): Note first that (iii) implies that
\[
\nu:= \sum_{x \in \Lambda} |c_x|^2 \delta_x
\]
is a finite positive measure on $\R^d$.

Let again $U \subseteq \R^d$ be such that $\Lambda$ is $U$-uniformly discrete. Let $\varphi \in \Cc^\infty(\R^d)$.
By a standard partition of unity argument, there exist $\varphi_1,\ldots,\varphi_k \in \Cc^\infty(\R^d)$ and $y_1,\ldots,y_k \in \R^d$ such that
\[
 \varphi =\sum_{j=1}^k\varphi_j \qquad \text{ with } \qquad
  \supp(\varphi_j) \subseteq y_j+U \ \text{ for all } 1 \leqslant j \leqslant k \,.
\]
By \cite[Lem.~5.8.3]{NS11} and its proof, we have
\[
| (\mu* \varphi_j)(y)| = (| \mu | * |\varphi_j|)(y) =\sum_{x \in \Lambda} |c_x|\, |\varphi_j(y-x)|
\]
with $|c_x||\varphi_j(y-x)|=0$,  for all $1 \leqslant j \leqslant k$ and all $y \in \R^d$, for all but at most one $x \in \Lambda$. In particular, this implies
\[
| (\mu* \varphi_j)(x) |^2  =\sum_{x \in \Lambda} |c_x|^2\, |\varphi_j(y-x)|^2 = \int_{\R^d} |\varphi_j(y-x)|^2\ \dd \nu(x)
\]
for all $1 \leqslant j \leqslant k$ and $y \in \R^d$.
Therefore, an application of Fubini's theorem gives
\begin{align*}
\int_{\R^d} | (\mu* \varphi_j)(y) |^2 \dd y
    &=\int_{\R^d} \int_{\R^d} |\varphi_j(y-x)|^2\ \dd \nu(x)\, \dd y\\
    &=\int_{\R^d} \int_{\R^d} |\varphi_j(y-x)|^2\ \dd y\ \dd \nu(x)
     =\| \varphi_j\|_2^2\int_{\R^d} 1\ \dd \nu(x) \\
    &=\| \varphi_j\|_2^2\,  \sum_{x \in \Lambda} |c_x|^2  < \infty \,.
\end{align*}
This completes the proof.
\end{proof}

The next corollary is an immediate consequence.

\begin{corollary}\label{cor234}
Let $\Lambda \subset \R^d$ be a uniformly discrete set, and let
\[
\mu= \sum_{x \in \Lambda} c_x \delta_x
\]
be a tempered measure such that $\sum_{x \in \Lambda } |c_x|^2 < \infty$. Then, its Fourier transform in the tempered distribution sense is an absolutely continuous measure.

In particular, if $\mu$ is Fourier transformable as a measure, $\widehat{\mu}$ is absolutely continuous.
\end{corollary}

\medskip

Fourier transformable measures with Meyer set support are of special interest in some branches of mathematics like Aperiodic Order. Therefore, we are going to list some consequences for this special class of measures.
To do so, we will recall the following result.

\begin{theorem}\label{T5}\cite[Thm.~4.1]{NS20a} \cite[Thm.~5.7]{NS21}
Let $\mu$ be a Fourier transformable measure supported inside a Meyer set $\Gamma \subset \R^d$. Then, there exists some Meyer set $\Lambda \subset \R^d$ and some function $c: \Lambda \to \C$ such that
\[
\gamma_{\operatorname{0a}}= \sum_{x \in \Lambda} c(x)\, \delta_x \,.
\]  \qed
\end{theorem}

By combining this result with Theorem~\ref{T4} and \cite[Cor.~35]{SS3}, we obtain the following consequence.

\begin{theorem}\label{T6}
Let $\mu$ be a Fourier transformable measure supported inside a Meyer set $\Gamma \subset \R^d$. Then, there exists some Meyer set $\Lambda \subset \R^d$ and some bounded function $c: \Lambda \to \C$ such that
\begin{itemize}
\item[(a)] $\mu_{\operatorname{0a}}= \sum_{x \in \Lambda} c(x)\, \delta_x$.
\item[(b)] $\lim_{ \|x \| \to \infty } |c(x)| =0$.
\item[(c)] If there exists some $f \in L^1_{\operatorname{loc}}(\R^d) \cap L^\infty(\R^d)$ such that $\left(\widehat{\mu}\right)_{\operatorname{ac}}=f \lm$, then
\[
\sum_{x \in \Lambda} |c(x)|^2 < \infty \,.
\]
\end{itemize}
\end{theorem}
\begin{proof} (a) This follows from Theorem~\ref{T5}.

\medskip

\noindent (b) This is a consequence of \cite[Cor.~35]{SS3}.

\medskip

\noindent (c) By Corollary~\ref{Cor212}, we have $f \in \mathcal{WLS}_2(\R^d)$. The claim follows now from Theorem~\ref{T4}.
\end{proof}

\section*{Acknowledgements}

This work was supported by the German Research Foundation (DFG) via research grant 415818660 as well as by the CRC 1283 at Bielefeld University
(TS), and by the Natural Sciences and Engineering Council of Canada
(NSERC), via grant 2020-00038 (NS).

\end{document}